\documentclass[10pt]{amsart}
\usepackage[psamsfonts]{amssymb}
\usepackage{amsthm,amsmath}

\title{Quasi-polynomial functions over bounded distributive lattices}

\author{Miguel Couceiro}
\address{Mathematics Research Unit, University of Luxembourg \\
162A, avenue de la Fa\"{\i}encerie, L-1511 Luxembourg, Luxembourg } \email{miguel.couceiro[at]uni.lu }

\author{Jean-Luc Marichal}
\address{Mathematics Research Unit, University of Luxembourg \\
162A, avenue de la Fa\"{\i}encerie, L-1511 Luxembourg, Luxembourg } \email{jean-luc.marichal[at]uni.lu }

\date{July 9, 2010}

\begin{document}

\theoremstyle{plain}
\newtheorem{theorem}{Theorem}
\newtheorem{lemma}[theorem]{Lemma}
\newtheorem{proposition}[theorem]{Proposition}
\newtheorem{corollary}[theorem]{Corollary}
\newtheorem{fact}[theorem]{Fact}
\newtheorem*{claim}{Claim}
\newtheorem*{main}{Main Theorem}

\theoremstyle{definition}
\newtheorem{definition}[theorem]{Definition}
\newtheorem{example}[theorem]{Example}

\theoremstyle{remark}
\newtheorem{remark}{Remark}

\newcommand{\N}{\mathbb{N}}                     
\newcommand{\R}{\mathbb{R}}                     
\newcommand{\co}[1]{\ensuremath{\overline{#1}}}
\newcommand{\vect}[1]{\ensuremath{\mathbf{#1}}} 
\def\med{\mathop{\rm med}\nolimits}

\begin{abstract}
In \cite{CouMar3} the authors introduced the notion of quasi-polynomial function as being a mapping $f\colon X^n\to X$ defined and valued on a
bounded chain $X$ and which can be factorized as $f(x_1,\ldots,x_n)=p(\varphi(x_1),\ldots,\varphi(x_n))$, where $p$ is a polynomial function
(i.e., a combination of variables and constants using the chain operations $\wedge$ and $\vee$) and $\varphi$ is an order-preserving map. In the
current paper we study this notion in the more general setting where the underlying domain and codomain sets are, possibly different, bounded
distributive lattices, and where the inner function is not necessarily order-preserving. These functions appear naturally within the scope of
decision making under uncertainty since, as shown in this paper, they subsume overall preference functionals associated with Sugeno integrals
whose variables are transformed by a given utility function. To axiomatize the class of quasi-polynomial functions, we propose several
generalizations of well-established properties in aggregation theory, as well as show that some of the characterizations given in \cite{CouMar3}
still hold in this general setting. Moreover, we investigate the so-called transformed polynomial functions (essentially, compositions of unary
mappings with polynomial functions) and show that, under certain conditions, they reduce to quasi-polynomial functions.
\end{abstract}



\maketitle

\section{Introduction}

When we need to summarize, fuse, or merge a set of values into a single one, we usually make use of a so-called aggregation function, e.g., a
mean or an averaging function. The need to aggregate values in a meaningful way has become more and more present in an increasing number of
areas not only of mathematics or physics, but especially in applied fields such as engineering, computer science, and economical and social
sciences. Various aggregation functions have been proposed in the literature, thus giving rise to the growing theory of aggregation which
proposes, analyzes, and characterizes aggregation function classes. For recent references, see Beliakov et al.~\cite{BelPraCal07} and Grabisch
et al.~\cite{GraMarMesPap09}.

Among noteworthy aggregation functions is the (discrete) Sugeno integral, which was introduced by Sugeno~\cite{Sug74,Sug77} as a way to compute
the average of a function with respect to a nonadditive measure. Since its introduction, the Sugeno integral has been thoroughly investigated
and is now considered as one of the most relevant aggregation functions in the qualitative setting of ordinal information (e.g., when the values
to be aggregated are simply defined on a chain without further structure). For general background, see also the edited book \cite{GraMurSug00}.

As it was observed in \cite{Mar01}, Sugeno integrals can be regarded as certain (lattice) polynomial functions, that is, functions which can be
obtained as combinations of variables and constants using the lattice operations $\wedge$ and $\vee$. More precisely, given a bounded chain $X$,
the Sugeno integrals are exactly those polynomial functions $f\colon X^n\to X$ which are idempotent, that is, satisfying $f(x,\ldots,x)=x$. This
convenient description made it possible to naturally extend the definition of the Sugeno integrals to the case when $X$ is a bounded
distributive lattice (see \cite{Mar09}) and to derive several axiomatizations of this class (as idempotent polynomial functions); see
\cite{CouMar1,CouMar0}.

In many applications, the values to be aggregated are first to be transformed by an order-preserving unary function $\varphi\colon X\to Y$ so
that the transformed values (which are usually real numbers) can be aggregated in a meaningful way by a function $g\colon Y^n\to Y$. The
resulting composed function $f\colon X^n\to Y$ is then defined as $f=g\circ\varphi$, that is,
\begin{equation}\label{eq:OverPrefFun}
f(x_1,\ldots,x_n)=g(\varphi(x_1),\ldots,\varphi(x_n)).
\end{equation}
Such an aggregation model is used for instance in decision under uncertainty, where $\varphi$ is called a utility function and $f$ an overall
preference functional. It is also used in multi-criteria decision making where the criteria are commensurate (i.e., expressed in a common
scale). For a recent reference, see the edited book \cite{BouDubPraPir09}.

This aggregation model has also been investigated in a purely ordinal decision setting, where $X$ and $Y$ are bounded chains and $g\colon Y^n\to
Y$ is a Sugeno integral or a polynomial function; see for instance \cite{DubFarPraSab09,DubMarPraRouSab01}. In the special case when $X=Y$, the
corresponding compositions (\ref{eq:OverPrefFun}), which we call quasi-polynomial functions, were recently investigated and characterized by the
authors as solutions of certain functional equations and in terms of necessary and sufficient conditions which have natural interpretations in
decision making and aggregation theory; see \cite{CouMar3}.

In this paper, after recalling the basic concepts in lattice theory as well as few well-known results concerning polynomial functions (Sections
2 and 3), we investigate the quasi-polynomial functions when considered in the more general setting where the underlying domain and codomain
sets are, possibly different, bounded distributive lattices, and where the inner unary functions are not necessarily order-preserving.
Moreover, we show that some axiomatizations of the class of quasi-polynomial functions given in \cite{CouMar3} still hold in this more general
setting and, under certain assumptions, we propose further characterizations of this class by necessary and sufficient conditions given in terms
of generalizations of well-established properties in aggregation theory. These results also lead to new characterizations of the class of
polynomial functions (Section 4). Finally, we introduce the concept of transformed polynomial function and show that, under certain conditions,
this notion is subsumed by that of quasi-polynomial function (Section 5).

\section{Basic notions and terminology}

Throughout this paper, let $X$ be an arbitrary bounded distributive lattice with lattice operations $\wedge$ and $\vee$, and with least and
greatest elements $0_X$ and $1_X$, respectively, where the subscripts may be omitted when the underlying lattice is clear from the context. A
\emph{chain} is simply a lattice $X$ such that, for every $a,b\in X$, we have $a\leqslant b$ or $b\leqslant a$. A subset $S$ of a lattice $X$ is
said to be \emph{convex} if, for every $a,b\in S$ and every $c\in X$ such that $a\leqslant c\leqslant b$, we have $c\in S$. For any subset
$S\subseteq X$, we denote by $\co S$ the convex hull of $S$, that is, the smallest convex subset of $X$ containing $S$.  For any integer
$n\geqslant 1$, we set $[n]=\{1,\ldots,n\}$.

The Cartesian product $X^n$ can be as well regarded as a bounded distributive lattice by defining $\wedge$ and $\vee$ componentwise, i.e.,
\begin{eqnarray*}
(a_1,\ldots,a_n)\wedge (b_1,\ldots,b_n) &=& (a_1\wedge b_1, \ldots , a_n\wedge b_n),\\
(a_1,\ldots,a_n)\vee (b_1,\ldots,b_n) &=& (a_1\vee b_1, \ldots , a_n\vee b_n).
\end{eqnarray*}
We denote the elements of $X$ by lower case letters $a,b,c,\ldots$, and the elements of $X^n$ by bold face letters
$\vect{a},\vect{b},\vect{c},\ldots$. We also use $\vect{0}$ and $\vect{1}$ to denote the least element and the greatest element, respectively,
of $X^n$ and we denote by $\vect{e}$ any $n$-tuple whose components are either $0$ or $1$, regardless of the underlying lattice. For $k\in [n]$,
$c\in X$, and $\vect{x}\in X^n$, we use $\vect{x}^c_k$ to denote the $n$-tuple whose $i$th component is $c$, if $i=k$, and $x_i$, otherwise. We
also let $\vect{x}\wedge c=(x_1\wedge c,\ldots,x_n\wedge c)$ and $\vect{x}\vee c=(x_1\vee c,\ldots,x_n\vee c)$, and denote by $[\vect{x}]_c$
(resp.\ $[\vect{x}]^c$) the $n$-tuple whose $i$th component is $0$ (resp. $1$), if $x_i\leqslant c$ (resp.\ $x_i\geqslant c$), and $x_i$,
otherwise.

Let $Y$ be an arbitrary bounded distributive lattice, possibly different from $X$. A mapping $\varphi\colon X\to Y$ is said to be a
\emph{lattice homomorphism} if it preserves the lattice operations, i.e.,
$$
\varphi(a\wedge_X b)=\varphi(a)\wedge_Y \varphi(b)\quad \textrm{and} \quad \varphi(a\vee_X b)=\varphi(a)\vee_Y \varphi(b).
$$
With no danger of ambiguity, we omit the subscripts $X$ and $Y$. For further background on lattice theory, see, e.g.,
\cite{BurSan81,Grae03,Rud01}.

The \emph{range} of a function $f\colon X^n\rightarrow Y$ is defined by $\mathcal{R}_f=\{f(\vect{x}) : \vect{x}\in X^n\}$. The \emph{diagonal
section} of $f$ is the unary function $\delta_f\colon X\to Y$ defined by $\delta_f(x)=f(x,\ldots,x)$. A function $f\colon X^n\rightarrow Y$ is
said to be \emph{order-preserving} (resp.\ \emph{order-reversing}) if, for every $\vect{a}, \vect{b}\in X^n$ such that
$\vect{a}\leqslant\vect{b}$, we have $f(\vect{a})\leqslant f(\vect{b})$ (resp.\ $f(\vect{a})\geqslant f(\vect{b})$). By a \emph{monotone}
function we simply mean an order-preserving or order-reversing function.
 As a typical example, we have the ternary median function $\med \colon X^3\rightarrow X$ which is given by
$$
\med(x_1,x_2,x_3) = (x_1 \wedge x_2)\vee (x_2 \wedge x_3)\vee (x_3 \wedge x_1).
$$
For any integer $m\geqslant 1$, any vector $\vect{x}\in X^m$, and any function $f\colon X^{n}\rightarrow X$, we define
$\langle\vect{x}\rangle_f\in X^m$ as the $m$-tuple $\langle\vect{x}\rangle_f=\med(f(\vect{0}),\vect{x},f(\vect{1}))$, where the right-hand side
median is taken componentwise. We then clearly have
\begin{equation} \label{eq:xxx1}
\langle\vect{x}\rangle_f=\langle\vect{x}\rangle_{\delta_f}.
\end{equation}
For $\varphi\colon X\to Y$ and $\vect{x}\in X^n$, we set $\varphi(\vect{x})=(\varphi(x_1),\ldots,\varphi(x_n))$.

Given a function $f\colon X^n\to Y$, we define the function $\hat f\colon\{0,1\}^n\to Y$ as
\begin{equation} \label{eq:hat}
\hat f(\vect{e})=\bigvee_{\textstyle{e'_i\in\{0,1\}\atop i\in [n]\, :\, e_i=1}}\bigwedge_{\textstyle{e'_i\in\{0,1\}\atop i\in [n]\, :\, e_i=0}}
f(\vect{e}'),\qquad \vect{e}\in\{0,1\}^n.
\end{equation}
Note that, for each fixed $\vect{e}\in\{0,1\}^n$, from among the $2n$ operations $\vee$ and $\wedge$ only $n$ are considered in (\ref{eq:hat}),
for if $e_i=0$ (resp.\ $e_i=1$) then the corresponding $\vee_{e'_i\in\{0,1\}}$ (resp.\ $\wedge_{e'_i\in\{0,1\}}$) is ignored and only
$\wedge_{e'_i\in\{0,1\}}$ (resp.\ $\vee_{e'_i\in\{0,1\}}$) is considered. In particular, we have that
$$
\hat f(\vect{0})=\underset{\vect{e}'\in\{0,1\}^n}{\bigwedge} f(\vect{e}')\quad\mbox{and}\quad\hat
f(\vect{1})=\underset{\vect{e}'\in\{0,1\}^n}{\bigvee} f(\vect{e}').
$$
For instance, if $n=2$, we have
\begin{equation}\label{eq:f-hat-2}
\hat f(0,1) = \big(f(0,0)\wedge f(1,0)\big)\vee\big(f(0,1)\wedge f(1,1)\big).
\end{equation}
If $f$ is order-preserving then $\hat f=f|_{\{0,1\}^n}$. However, note that $\hat f$ is always order-preserving.

\begin{remark}\label{remark-hat}
Note that there are $(2n)!$ ways of rearranging the $\vee$ and $\wedge$ in expression (\ref{eq:hat}) and thus $(2n)!$ ways of constructing such
an order-preserving function from a given $f\colon X^n\to Y$, each of which producing a possibly different order-preserving function. For
instance,  let $f\colon \{0,1\}^2\to  \{0,1\}$ be the Boolean addition, i.e., addition modulo 2. Then, as defined in (\ref{eq:hat}), $\hat f$ is
the binary conjunction $\wedge$. On the other hand, if $\hat f$ had been defined as
\begin{equation} \label{eq:hatp}
\hat f(\vect{e})=\bigwedge_{\textstyle{e'_i\in\{0,1\}\atop i\in [n]\, :\, e_i=0}} \bigvee_{\textstyle{e'_i\in\{0,1\}\atop i\in [n]\, :\, e_i=1}}
f(\vect{e}'),\qquad \vect{e}\in\{0,1\}^n,
\end{equation}
then $\hat f$ would have been the binary disjunction $\vee$. To avoid such ambiguities, we only refer to $\hat f$ as given in (\ref{eq:hat}).
\end{remark}

\section{General background on polynomial functions}

In this subsection we recall some important results on polynomial functions that will be needed hereinafter. For further background, we refer
the reader to \cite{BurSan81,CouMar1,CouMar2,CouMar0,Goo67,Grae03,Rud01}.

A (\emph{lattice}) \emph{polynomial function} on a lattice $X$ is any map  $p\colon X^n\to X$ which can be obtained as a composition of the
lattice operations $\wedge$ and $\vee$, the projections $\vect{x}\mapsto x_i$, $i\in [n]$, and the constant functions $\vect{x}\mapsto c$, $c\in
X$.

\begin{fact}\label{nondecreasing}
Every polynomial function $p\colon X^n\to X$ is order-preserving and satisfies $\delta_p(c)=\langle c\rangle_p$ for every $c\in X$. In
particular $\delta_p(c)=c$ for every $c\in \mathcal{R}_p$.
\end{fact}

Polynomial functions $p\colon X^n\to X$ satisfying $\mathcal{R}_p=X$ are referred to as (\emph{discrete}) \emph{Sugeno integrals}. By
Fact~\ref{nondecreasing}, Sugeno integrals are exactly those polynomial functions $q\colon X^n\to X$ satisfying $q(\vect{0})=0$ and
$q(\vect{1})=1$, and thus every polynomial function $p\colon X^n\to X$ is of the form  $p=\langle q\rangle_p$ for a suitable Sugeno integral
$q\colon X^n\to X$ (see \cite{Mar09}). The following result is due to Goodstein~\cite{Goo67}.

\begin{proposition}\label{Goodstein}
Every polynomial function is completely determined by its restriction to $\{0,1\}^n$. Moreover, a function $g\colon \{0,1\}^n\rightarrow X$ can
be extended to a polynomial function $f\colon X^{n}\rightarrow X$ if and only if it is order-preserving, and in this case the extension is
unique.
\end{proposition}

As observed by Goodstein~\cite{Goo67}, polynomial functions on bounded distributive lattices have neat normal form representations. To this
extent, for each $I\subseteq [n]$, let $\vect{e}_I$ be the element of $X^n$ whose $i$th component is $1$, if $i\in I$, and $0$, otherwise. Let
$\alpha_f\colon 2^{[n]}\to X$ and $\beta_f\colon 2^{[n]}\to X$ be the functions defined by $\alpha_f(I)=f(\vect{e}_I)$ and
$\beta_f(I)=f(\vect{e}_{[n]\setminus I}),$ respectively.

\begin{proposition}
A function $f\colon X^{n}\rightarrow X$ is a polynomial function if and only if
\begin{equation}\label{eq:DNF45}
f(\vect{x})=\bigvee_{I\subseteq [n]}\Big(\alpha_f(I)\wedge \bigwedge_{i\in I} x_i\Big) \quad \mbox{or} \quad f(\vect{x})=\bigwedge_{I\subseteq
[n]}\Big(\beta_f(I)\vee \bigvee_{i\in I} x_i\Big).
\end{equation}
\end{proposition}
The expressions given in (\ref{eq:DNF45}) are usually referred to as the \emph{disjunctive normal form} (DNF) representation and the
\emph{conjunctive normal form} (CNF) representation, respectively, of the polynomial function $f$.

In the sequel, we will make use of some characterizations  of polynomial functions obtained in \cite{CouMar1,CouMar0}. For the sake of
self-containment, we recall these results.

Let $S\subseteq X$. A function $f\colon X^{n}\rightarrow X$ is said to be
\begin{itemize}
\item \emph{$S$-idempotent} if  $\delta_f(c)=c$, for every $c\in S$.

\item \emph{$\wedge_S$-homogeneous} if $f(\vect{x}\wedge c) = f(\vect{x})\wedge c$ for all $\vect{x}\in X^n$ and $c\in S$.

\item \emph{$\vee_S$-homogeneous} if $f(\vect{x}\vee c) = f(\vect{x})\vee c$ for all $\vect{x}\in X^n$ and $c\in S$.

\item \emph{median decomposable} if $f(\vect{x})=\med\big(f(\vect{x}^{0}_{k}), x_k, f(\vect{x}^{1}_{k})\big)$ for all $\vect{x}\in X^n$ and
$k\in [n]$.
\end{itemize}

\begin{theorem} \label{mainChar}
Let $f\colon X^{n}\rightarrow X$ be a function. The following conditions are equivalent:
\begin{enumerate}
\item[(i)] $f$ is a polynomial function.

\item[(ii)] $f$ is median decomposable.

\item[(iii)] $f$ is order-preserving, and $\wedge_{\overline{\mathcal{R}}_f}$- and $\vee_{\overline{\mathcal{R}}_f}$-homogeneous.
\end{enumerate}
\end{theorem}

In the case when $X$ is a chain, Theorem~\ref{mainChar} can be drastically refined for the conditions provided only need to be verified on
$n$-tuples of a certain prescribed type (see \cite{CouMar2}). Moreover, further characterizations are available, and given in terms of
conditions of somewhat different flavor, as the following theorem illustrates.

Let $\sigma$ be a permutation on $[n]$. The \emph{standard simplex} of $X^n$ associated with $\sigma$ is the subset $X^n_\sigma\subset X^n$
defined by $X^n_\sigma =\{\vect{x}\in X^n\colon x_{\sigma (1)}\leqslant x_{\sigma (2)}\leqslant \cdots\leqslant x_{\sigma (n)}\}.$

\begin{theorem}[{\cite{CouMar2}}]\label{theorem:WLP-comonot}
Let $X$ be a bounded chain. A function $f\colon X^{n}\rightarrow X$ is a polynomial function if and only if it is
$\overline{\mathcal{R}}_f$-idempotent, and comonotonic minitive and comonotonic maxitive, that is, for every permutation $\sigma$ on $[n]$, and
every $\vect{x},\vect{x}'\in X^n_\sigma$,
$$
f(\vect{x}\wedge \vect{x}') = f(\vect{x})\wedge f(\vect{x}')  \quad \mbox{and} \quad f(\vect{x}\vee \vect{x}') = f(\vect{x})\vee f(\vect{x}'),
\quad \mbox{resp.}
$$
\end{theorem}

\section{Quasi-polynomial functions and quasi-Sugeno integrals}

The notions of polynomial function and Sugeno integral can be naturally extended to functions defined on a bounded distributive lattice $X$ and
valued on a possibly different bounded distributive lattice $Y$ via the concepts of quasi-polynomial function and quasi-Sugeno integral.

\begin{definition}\label{de:QP-QS}
We say that a function $f\colon X^n\to Y$ is a \emph{quasi-polynomial function} (resp.\ a \emph{quasi-Sugeno integral}) if there exist a
polynomial function (resp.\ a Sugeno integral) $p\colon Y^n\to Y$ and a unary function $\varphi\colon X\to Y$, satisfying
$\varphi=\langle\varphi\rangle_{\varphi}$, such that $f=p\circ\varphi$, that is,
\begin{equation}\label{eq:QuasiPol2}
f(x_1,\ldots,x_n)=p(\varphi(x_1),\ldots,\varphi(x_n)).
\end{equation}
\end{definition}

\begin{remark}
The condition $\varphi=\langle\varphi\rangle_{\varphi}$ in Definition~\ref{de:QP-QS} simply ensures that the values of $\varphi$ are not too
scattered with respect to $\varphi(0)$ and $\varphi(1)$. More precisely, the values of $\varphi$ lie in the interval
$[\varphi(0)\wedge\varphi(1),\varphi(0)\vee\varphi(1)]$. In the case when $\varphi$ is monotone, this condition is satisfied since it translates
into saying that $\varphi(0)\leqslant\varphi(x)\leqslant\varphi(1)$ or $\varphi(1)\leqslant\varphi(x)\leqslant\varphi(0)$.
\end{remark}

It is easy to see that the functions $p$ and $\varphi$ in (\ref{eq:QuasiPol2}) need not be unique. For instance, if $f$ is a constant $c\in Y$,
then we could choose $p\equiv c$ and $\varphi$ arbitrarily, or $p$ idempotent and $\varphi\equiv c$. We make use of the following lemma to show
that we can always choose $\delta_f$ for $\varphi$.

\begin{lemma}\label{lemma:45474786}
Every polynomial function $p\colon X^n\to X$ satisfies $p(\vect{x})=p(\langle\vect{x}\rangle_p)$, $p(\vect{x}\vee c)=p(\vect{x})\vee\langle
c\rangle_p$, and $p(\vect{x}\wedge c)=p(\vect{x})\wedge\langle c\rangle_p$ for every $\vect{x}\in X^n$ and every $c\in X$.
\end{lemma}

\begin{proof}
Since $p$ is $\vee_{\overline{\mathcal{R}}_p}$- and $\wedge_{\overline{\mathcal{R}}_p}$-homogeneous, we have
$p(\vect{x})=p(\langle\vect{x}\rangle_p)$. We then have $p(\vect{x})\vee\langle c\rangle_p=p(\langle\vect{x}\rangle_p)\vee\langle
c\rangle_p=p(\langle\vect{x}\vee c\rangle_p)=p(\vect{x}\vee c)$. The last identity can be proved dually.
\end{proof}

\begin{proposition}\label{prop:ff56}
Let $p\colon Y^n\to Y$ be a polynomial function, $\varphi\colon X\to Y$ a unary function, and set $f=p\circ\varphi$. Then, we have
$\delta_f=\langle\varphi\rangle_p$ and $f=p\circ\delta_f$. In particular, if $p$ is a Sugeno integral, then we have $\varphi=\delta_f$.
\end{proposition}

\begin{proof}
By Fact~\ref{nondecreasing}, we have $\delta_f=\delta_p\circ\varphi=\langle\varphi\rangle_p$ and hence, by Lemma~\ref{lemma:45474786},
$f=p\circ\varphi=p\circ\langle\varphi\rangle_p=p\circ\delta_f$. The second part follows immediately.
\end{proof}

To explicitly describe all possible factorizations of quasi-polynomial functions into compositions of polynomial functions with unary maps, we
shall make use of the following useful tool.

\begin{lemma}\label{lemma:jff8sdf}
Let $p\colon Y^n\to Y$ be a polynomial function, $\varphi\colon X\to Y$ a unary function, and set $f=p\circ\varphi$. Then $\langle
p(\vect{x})\rangle_f=p(\langle\vect{x}\rangle_{\varphi})$ for every $\vect{x}\in Y^n$. In particular, if
$\varphi=\langle\varphi\rangle_{\varphi}$, then $f=\langle f\rangle_{f}$.
\end{lemma}

\begin{proof}
For every $\vect{x}\in Y^n$, we have
\begin{eqnarray*}
\langle p(\vect{x})\rangle_f &=& \med\big(f(\vect{0}),p(\vect{x}),f(\vect{1)}\big)\\
&=& \big(f(\vect{0})\wedge f(\vect{1})\big)\vee\big(p(\vect{x})\wedge\big(f(\vect{0})\vee f(\vect{1})\big)\big)\\
&=& \big\langle\varphi(0)\wedge\varphi(1)\big\rangle_p\vee
\big(p(\vect{x})\wedge\big\langle\varphi(0)\vee\varphi(1)\big\rangle_p\big)\qquad \mbox{(by Proposition~\ref{prop:ff56})}\\
&=& p\big((\varphi(0)\wedge\varphi(1))\vee(\vect{x}\wedge(\varphi(0)\vee\varphi(1))\big)\qquad \mbox{(by Lemma~\ref{lemma:45474786})}\\
&=& p(\langle\vect{x}\rangle_{\varphi}).
\end{eqnarray*}
In particular, if $\varphi=\langle\varphi\rangle_{\varphi}$, then $\langle f\rangle_{f}=\langle p\circ\varphi\rangle_f =
p\circ\langle\varphi\rangle_{\varphi}=p\circ\varphi=f$.
\end{proof}

The following proposition provides alternative factorizations of a quasi-polynomial function. Let $p_f\colon Y^n\to Y$ be the unique polynomial
function extending $\hat f$ (see Proposition~\ref{Goodstein}).

\begin{proposition}\label{prop:DescrPP545664}
Let $f\colon X^n\to Y$ be a quasi-polynomial function, $p\colon Y^n\to Y$ a polynomial function, and $\varphi\colon X\to Y$ a unary function
satisfying $\varphi=\langle\varphi\rangle_{\varphi}$. Then we have $f=p\circ\varphi$ if and only if $p_f=\langle p\rangle_f$ and
$\delta_f=\langle\varphi\rangle_p$. In particular, we have $f=p_f\circ\delta_f$.
\end{proposition}

\begin{proof}
We first establish the following claim.

\begin{claim}
If $f=p\circ\varphi$, then $\delta_f=\langle\varphi\rangle_p$, $f=p\circ\delta_f$, and $\hat f=\langle p\rangle_f|_{\{0,1\}^n}$.
\end{claim}

\begin{proof}[Proof of the claim]
The first two formulas follow immediately from Proposition~\ref{prop:ff56}. By Lemma~\ref{lemma:jff8sdf}, for every $\vect{e}\in\{0,1\}^n$, we
have $\langle p(\vect{e})\rangle_f=p(\langle\vect{e}\rangle_{\varphi})$, where $\langle\vect{e}\rangle_{\varphi}$ is the $n$-tuple whose $i$th
component is $\varphi(0)\vee\varphi(1)$, if $e_i=1$, and $\varphi(0)\wedge\varphi(1)$, if $e_i=0$. Since $p$ preserves $\vee$ and $\wedge$
componentwise, when expanding $p(\langle\vect{e}\rangle_{\varphi})$ in terms of $\vee$ and then in terms of $\wedge$, we obtain $\langle
p(\vect{e})\rangle_f=\hat f(\vect{e})$. To illustrate, if $\vect{e}=(0,1)$, we have
\begin{eqnarray*}
\langle p(0,1)\rangle_f &=& p\big(\varphi(0)\wedge\varphi(1),\varphi(0)\vee\varphi(1)\big)\\
&=& p\big(\varphi(0)\wedge\varphi(1),\varphi(0)\big)\vee p\big(\varphi(0)\wedge\varphi(1),\varphi(1)\big)\\
&=& \big(p(\varphi(0),\varphi(0))\wedge p(\varphi(1),\varphi(0))\big)\vee\big(p(\varphi(0),\varphi(1))\wedge p(\varphi(1),\varphi(1))\big)\\
&=& \big(f(0,0)\wedge f(1,0)\big)\vee\big(f(0,1)\wedge f(1,1)\big)\\
&=& \hat f(0,1)\qquad (\mbox{by (\ref{eq:f-hat-2})}).\qedhere
\end{eqnarray*}
\end{proof}

(Necessity) Suppose that $f=p\circ\varphi$. By the claim, we have $\delta_f=\langle\varphi\rangle_p$ and $p_f|_{\{0,1\}^n}=\hat f=\langle
p\rangle_f|_{\{0,1\}^n}$. This completes the proof since two polynomial functions having the same values on $\{0,1\}^n$ coincide by
Proposition~\ref{Goodstein}.

(Sufficiency) Since $f$ is a quasi-polynomial function, there exist a polynomial function $q\colon Y^n\to Y$ and a unary function $\psi\colon
X\to Y$, satisfying $\psi=\langle\psi\rangle_{\psi}$, such that $f=q\circ\psi$. By the claim, we then have $\delta_f=\langle\psi\rangle_q$,
$f=q\circ\delta_f$, and
\begin{equation}\label{eq:sdf87fff}
\langle q\rangle_f|_{\{0,1\}^n}=\hat f=p_f|_{\{0,1\}^n}=\langle p\rangle_f|_{\{0,1\}^n},
\end{equation}
which, by Proposition~\ref{Goodstein} and (\ref{eq:sdf87fff}), implies that
\begin{equation}\label{eq:sdf87fff1}
\langle q\rangle_f=\langle p\rangle_f.
\end{equation}

Therefore, we have
\begin{eqnarray*}
p\circ\varphi &=& p\circ\langle\varphi\rangle_{\varphi}\\
&=& \langle p\circ\varphi\rangle_{p\circ\varphi} \qquad \mbox{(by Lemma~\ref{lemma:jff8sdf})}\\
&=& \langle p\circ\varphi\rangle_{\delta_{p\circ\varphi}} \qquad \mbox{(by (\ref{eq:xxx1}))}\\
&=& \langle p\circ\varphi\rangle_{\langle\varphi\rangle_p} \qquad \mbox{(by Proposition~\ref{prop:ff56})}\\
&=& \big\langle p\circ\langle\varphi\rangle_p\big\rangle_{\langle\varphi\rangle_p} \qquad \mbox{(by Lemma~\ref{lemma:45474786})}\\
&=& \langle p\circ\delta_f\rangle_{\delta_f}\\
&=& \langle p\circ\delta_f\rangle_f \qquad \mbox{(by (\ref{eq:xxx1}))}\\
&=& \langle p\rangle_f\circ\delta_f\\
&=& \langle q\rangle_f\circ\delta_f \qquad \mbox{(by (\ref{eq:sdf87fff1}))}\\
&=& \langle q\circ\delta_f\rangle_f\\
&=& \langle f\rangle_f\\
&=& f\qquad \mbox{(by Lemma~\ref{lemma:jff8sdf}).}\qedhere
\end{eqnarray*}
\end{proof}

\begin{remark}\label{remark-hat2}
As mentioned in Remark~\ref{remark-hat}, there are several ways of constructing an order-preserving function from a given $f\colon X^n\to Y$,
e.g., using an expression in DNF (as in (\ref{eq:hat})) or using an expression in CNF (as in (\ref{eq:hatp})). Even though, in general,
different constructions may lead to different order-preserving functions (as in the case of the Boolean sum), by the claim it follows that, if
$f$ is a quasi-polynomial function, then all possible rearrangements of $\vee$ and $\wedge$ produce the same order-preserving function $\hat f$.
However, the converse statement does not hold, since there are order-preserving functions which do not constitute quasi-polynomial functions.
\end{remark}

%
%

Recall that every polynomial function $p\colon X^n\to X$ can be represented as $\langle q\rangle_p$ for some Sugeno integral $q\colon X^n\to X$.
Using this fact, we obtain the following result.

\begin{proposition}\label{prop:Sugeno4565}
A function $f\colon X^n\to Y$ is a quasi-polynomial function if and only if it is a quasi-Sugeno integral.
\end{proposition}

\begin{proof}
Since Sugeno integrals are polynomial functions, it follows that every quasi-Sugeno integral is a quasi-polynomial function. For the converse,
let $f\colon X^n\to Y$ be a quasi-polynomial function as described in (\ref{eq:QuasiPol2}). Let $q\colon Y^n\to Y$ be a Sugeno integral such
that $p=\langle q\rangle_p$. Since $q$ is $\wedge_Y$- and $\vee_Y$-homogeneous (see Theorem~\ref{mainChar}), we have $q\circ \langle\varphi
\rangle_p=\langle q\circ\varphi \rangle_p=\langle q\rangle_p\circ\varphi =p\circ\varphi=f$, which shows that $f$ is a quasi-Sugeno integral.
\end{proof}

\begin{remark}
Proposition~\ref{prop:Sugeno4565} shows that monotone quasi-polynomial functions are of interest in the ordinal settings of decision making
since they coincide with monotone quasi-Sugeno integrals, which were characterized as overall preference functionals for instance in decision
under uncertainty; see \cite{DubMarPraRouSab01}.
\end{remark}

We now consider a number of characterizations of the class of quasi-polynomial functions. These characterizations are inspired from those
obtained by the authors \cite{CouMar3} in the special case when $X=Y$ is a bounded chain.

We say that a function $f\colon X^n\rightarrow Y$ is \emph{quasi-median decomposable} if, for every $\vect{x}\in X^n$ and every $k\in [n]$, we
have
\begin{equation}\label{quasimedianDecomp}
f(\vect{x})=\med\big(f(\vect{x}^{0}_{k}), \delta_f(x_k), f(\vect{x}^{1}_{k})\big).
\end{equation}
Note that every unary function $\varphi\colon X\to Y$ satisfying $\varphi=\langle\varphi\rangle_{\varphi}$ (in particular, every monotone unary
function) is quasi-median decomposable. The following theorem provides a characterization for quasi-polynomial functions in terms of
quasi-median decomposition.

\begin{theorem}\label{theorem:QuasiMedian}
A function $f\colon X^n\rightarrow Y$ is a quasi-polynomial function if and only if it is quasi-median decomposable.
\end{theorem}

\begin{proof}
To verify that the condition is sufficient, just observe that applying (\ref{quasimedianDecomp}) repeatedly to each variable of $f$ we can
straightforwardly obtain a representation of $f$ as $f=p\circ\delta_f$ for some polynomial function $p$. Moreover, we have
$\delta_f=\langle\delta_f\rangle_{\delta_f}$.

Conversely, suppose that $f\colon X^{n}\rightarrow Y$ is a quasi-polynomial function. By Proposition~\ref{prop:ff56} and
Lemma~\ref{lemma:jff8sdf}, we have $f=\langle f\rangle_f$ and there exists a polynomial function $p\colon Y^n\to Y$ such that
$f=p\circ\delta_f$. By Theorem~\ref{mainChar}, for every $\vect{x}\in X^n$ and every $k\in [n]$, we have
$$
f(\vect{x}) = \langle f(\vect{x})\rangle_f = \big\langle \med\big(p(\delta_f(\vect{x})^{0}_{k}), \delta_f(x_k),
p(\delta_f(\vect{x})^{1}_{k})\big)\big\rangle_f\, ,
$$
that is, since $\med$ and $\langle\cdot\rangle_f$ commute,
\begin{equation}\label{12345}
f(\vect{x}) = \med\big(\langle p(\delta_f(\vect{x})^{0}_{k})\rangle_f, \delta_f(x_k),\langle p(\delta_f(\vect{x})^{1}_{k})\rangle_f\big).
\end{equation}
However, by using Lemma~\ref{lemma:jff8sdf} twice, we obtain
\begin{eqnarray*}
\langle p(\delta_f(\vect{x})^{0}_{k})\rangle_f &=& p(\langle\delta_f(\vect{x})_k^0\rangle_{\delta_f})%
~=~ p\big(\big\langle\delta_f(\vect{x})_k^{\delta_f(0)\wedge\delta_f(1)}\big\rangle_{\delta_f}\big)\\
&=& \big\langle p\big(\delta_f(\vect{x})_k^{\delta_f(0)\wedge\delta_f(1)}\big)\big\rangle_f
\end{eqnarray*}
and, since $p$ preserves $\vee$ and $\wedge$ componentwise, we also have
$$
\langle p(\delta_f(\vect{x})^{0}_{k})\rangle_f = \langle f(\vect{x}_k^0)\wedge f(\vect{x}_k^1)\rangle_f = f(\vect{x}_k^0)\wedge f(\vect{x}_k^1).
$$
Similarly, we have $\langle p(\delta_f(\vect{x})^{1}_{k})\rangle_f = f(\vect{x}_k^0)\vee f(\vect{x}_k^1)$. Combining this with (\ref{12345}), we
see that $f$ is quasi-median decomposable.
\end{proof}

We say that a function $f\colon X^{n}\rightarrow Y$ is \emph{quasi-$\vee$-homogeneous} (resp.\ \emph{quasi-$\wedge$-homogeneous}) if for every
$\vect{x}\in X^n$ and $c\in X$, we have
$$
f(\vect{x}\vee c) = f(\vect{x})\vee  \delta_f(c) \qquad \textrm{(resp.\ $f(\vect{x}\wedge c) = f(\vect{x})\wedge \delta_f(c)$}).
$$
Note that, for every quasi-$\vee$-homogeneous (resp.\ quasi-$\wedge$-homogeneous) function $f\colon X^{n}\rightarrow Y$, the diagonal section
$\delta_f$ preserves $\vee$ (resp.\ $\wedge$).

A function $f\colon X^{n}\rightarrow Y$ is said to be \emph{horizontally $\wedge$-decomposable} (resp.\ \emph{horizontally $\vee$-decomposable})
if for every $\vect{x}\in X^n$ and $c\in X$, we have
$$
f(\vect{x}) = f(\vect{x}\vee c)\wedge f([\vect{x}]^c) \qquad \textrm{(resp.\ $f(\vect{x}) = f(\vect{x}\wedge c)\vee f([\vect{x}]_c)$}).
$$

By considering functions $f\colon X^{n}\rightarrow Y$ whose diagonal section $\delta_f$ is a lattice homomorphism (i.e., $\delta_f$ preserves
both $\vee$ and $\wedge$), we can further extend some of the characterizations presented in \cite{CouMar3} to the present setting. Note that any
of these preservation conditions immediately ensures the order-preservation of $\delta_f$, in which case, by Fact~\ref{nondecreasing} and
Proposition~\ref{prop:ff56}, if $f$ is a quasi-polynomial function, then it is necessarily order-preserving.

\begin{theorem}\label{theorem:QuasiHomoMedian}
Let $f\colon X^{n}\rightarrow Y$ be an order-preserving function whose diagonal section $\delta_f$ is a lattice homomorphism. The following are
equivalent:
\begin{enumerate}
\item[$(i)$] $f$ is a quasi-polynomial function. \item[$(ii)$] $f$ is quasi-$\wedge$-homogeneous and quasi-$\vee$-homogeneous. \item[$(iii)$]
$f$ is quasi-$\wedge$-homogeneous and horizontally $\vee$-decomposable. \item[$(iv)$] $f$ is horizontally $\wedge$-decomposable and
quasi-$\vee$-homogeneous.
\end{enumerate}
\end{theorem}

\begin{proof}
We show that $(i)\Rightarrow (ii)\Rightarrow(iii)\Rightarrow (i)$. The equivalence $(i)\Leftrightarrow (iv)$ follows dually.

So suppose that $f\colon X^{n}\rightarrow Y$ is a quasi-polynomial function whose diagonal section $\delta_f$ is a lattice homomorphism. By
Propositions~\ref{prop:ff56} and \ref{prop:Sugeno4565}, we may assume that $f=p\circ \delta_f$ for some Sugeno integral $p\colon Y^n\to Y$.
Since $p$ is $\wedge_Y$-homogeneous, we have for every $c\in X$,
\begin{eqnarray*}
f(\vect{x}\wedge c) &=& p(\delta_f(\vect{x}\wedge c)) ~=~ p(\delta_f(\vect{x})\wedge \delta_f(c)) ~=~ p(\delta_f(\vect{x}))\wedge \delta_f(c)\\
&=& f(\vect{x}) \wedge \delta_f(c),
\end{eqnarray*}
which shows that $f$ is quasi-$\wedge$-homogeneous. Dually, it can be shown that $f$ is quasi-$\vee$-homogeneous.

To see that $(ii)\Rightarrow (iii)$, suppose that $f\colon X^{n}\rightarrow Y$ satisfies $(ii)$. Since $f$ is order-preserving, for every  $c\in
X$, we have
\begin{eqnarray*}
f(\vect{x}\wedge c) \vee f([\vect{x}]_{c})
&=&  (f(\vect{x}) \wedge \delta_f(c)) \vee f([\vect{x}]_{c}) ~=~  f(\vect{x}) \wedge  (\delta_f(c) \vee f([\vect{x}]_{c}) )\\
&=&  f(\vect{x}) \wedge  f([\vect{x}]_{c} \vee c ) ~=~  f(\vect{x}),
\end{eqnarray*}
thus showing that $(iii)$ holds.

To show that $(iii)\Rightarrow (i)$, by Theorem~\ref{theorem:QuasiMedian}, it is enough to show that $(iii)$ implies that $f$ is quasi-median
decomposable. Let $\vect{x}\in X^n$ and take $k\in [n]$. Since $f$ is horizontally $\vee$-decomposable, we have
$$
f(\vect{x})=f(\vect{x}\wedge x_k)\vee f([\vect{x}]_{x_k}).
$$
By quasi-$\wedge$-homogeneity, we have
$$
f(\vect{x}\wedge x_k)=f(\vect{x}_k^{1}\wedge x_k)=f(\vect{x}_k^{1})\wedge \delta_f(x_k)
$$
and by the definition of $[\vect{x}]_{x_k}$, we have $f([\vect{x}]_{x_k})\leqslant f(\vect{x}_k^{0})$. Thus,
\begin{eqnarray*}
f(\vect{x}) &=& \big(f(\vect{x}_k^{0})\vee f(\vect{x})\big)\wedge f(\vect{x}_k^{1})%
~=~ \big(f(\vect{x}_k^{0})\vee (f(\vect{x}_k^{1})\wedge \delta_f(x_k))\big)\wedge f(\vect{x}_k^{1})\\
&=& f(\vect{x}_k^{0}) \vee \big(f(\vect{x}_k^{1})\wedge \delta_f(x_k)\big)%
~=~ \med \big(f(\vect{x}_k^{0}),\delta_f(x_k),f(\vect{x}_k^{1})\big),
\end{eqnarray*}
which completes the proof of the theorem.
\end{proof}

Let $Y$ be a bounded chain. We say that a function $f\colon X^{n}\rightarrow Y$ is \emph{quasi-comonotonic minitive} (resp.
\emph{quasi-comonotonic maxitive}) if
 for every permutation $\sigma$ on $[n]$ and every $\vect{x},\vect{x}'\in X^n$ such that $\delta_f (\vect{x}),\delta_f (\vect{x}') \in Y^n_\sigma$, we have
$$
f(\vect{x}\wedge \vect{x}') = f(\vect{x})\wedge f(\vect{x}')  \quad \mbox{(resp. $ f(\vect{x}\vee \vect{x}') = f(\vect{x})\vee f(\vect{x}')).$}
$$

By Theorem~\ref{theorem:WLP-comonot}, it follows that every quasi-polynomial function whose diagonal section $\delta_f$ is a lattice
homomorphism is quasi-comonotonic minitive and maxitive. Moreover, it is easy to verify that if a function is quasi-comonotonic minitive (resp.\
quasi-comonotonic maxitive), then it is quasi-$\wedge$-homogeneous (resp.\ quasi-$\vee$-homogeneous). Thus, using
Theorem~\ref{theorem:QuasiHomoMedian} $(ii)$, we obtain the following characterization of quasi-polynomial functions valued on bounded chains.


\begin{theorem}\label{theorem:QLP-comonot}
Let $X$ be a bounded distributive lattice, $Y$ a bounded chain, and $f\colon X^n\rightarrow Y$ an order-preserving function whose diagonal
section $\delta_f$ is a lattice homomorphism. Then $f$ is a quasi-polynomial function if and only if it is quasi-comonotonic minitive and
quasi-comonotonic maxitive.
\end{theorem}

\begin{remark}
Theorems~\ref{theorem:QuasiMedian}, \ref{theorem:QuasiHomoMedian}, and \ref{theorem:QLP-comonot} extend to the present setting some results
obtained in \cite{CouMar3} when $X=Y$ is a bounded chain. Order-preserving quasi-polynomial functions are also characterized in the latter case
by order-preservation and horizontally $\wedge$- and $\vee$-decompositions (see \cite[Theorem 11]{CouMar3}). It remains open whether this result
still holds in the present general setting.
\end{remark}

We now use Theorem~\ref{theorem:QuasiHomoMedian} to derive further characterizations of the class of polynomial functions. We first recall the
following lemma given in \cite{CouMar0}.

\begin{lemma}\label{lemma:IdEq653}
A unary function $f\colon X\to X$ is a polynomial function if and only if $f$ is a solution of the equation $f\circ f=f$, is a lattice
homomorphism, and has a convex range.
\end{lemma}

The following proposition provides conditions under which a quasi-polynomial function is a polynomial function.

\begin{proposition}\label{prop:5353452}
Let $f\colon X^n\to X$ be a quasi-polynomial function. Then $f$ is a polynomial function if and only if it is $\mathcal{R}_f$-idempotent and
$\delta_f$ is a lattice homomorphism with a convex range.
\end{proposition}

\begin{proof}
The necessity is straightforward (see \cite{CouMar0}). Let us prove the sufficiency. Since $f$ is $\mathcal{R}_f$-idempotent, we have
$\delta_f\circ\delta_f=\delta_f$ and, by Lemma~\ref{lemma:IdEq653}, $\delta_f$ is a polynomial function. Since $f$ is a quasi-polynomial
function, by Proposition~\ref{prop:ff56}, $f$ is a polynomial function.
\end{proof}

By combining Theorem~\ref{theorem:QuasiHomoMedian} with Proposition~\ref{prop:5353452}, we obtain the following characterizations of the class
of polynomial functions.

\begin{corollary}
Let $f\colon X^n\to X$ be an order-preserving and $\mathcal{R}_f$-idempotent function whose diagonal section $\delta_f$ is a lattice
homomorphism with a convex range. The following are equivalent:
\begin{enumerate}
\item[$(i)$] $f$ is a polynomial function. \item[$(ii)$] $f$ is quasi-$\wedge$-homogeneous and quasi-$\vee$-homogeneous. \item[$(iii)$] $f$ is
quasi-$\wedge$-homogeneous and horizontally $\vee$-decomposable. \item[$(iv)$] $f$ is horizontally $\wedge$-decomposable and
quasi-$\vee$-homogeneous.
\end{enumerate}
\end{corollary}

\section{Transformed polynomial functions and Sugeno integrals}

We have defined quasi-polynomial functions by considering polynomial functions whose variables are first transformed by a certain unary
function. Instead of transforming the variables, we could transform the polynomial function itself. This leads to the following definition.

\begin{definition}\label{de:Tr-PF}
We say that a function $f\colon X^n\to Y$ is a \emph{transformed polynomial function} (resp.\ \emph{a transformed Sugeno integral}) if there
exist a polynomial function (resp.\ a Sugeno integral) $p\colon X^n\to X$ and a function $\psi\colon X\to Y$ such that $f=\psi\circ p$, that is,
\begin{equation}\label{eq:Tr-pol7434}
f(x_1,\ldots,x_n)=\psi(p(x_1,\ldots,x_n)).
\end{equation}
\end{definition}


The functions $\psi$ and $p$ in (\ref{eq:Tr-pol7434}) need not be unique. For instance, if $f$ is a constant $c\in Y$, then we could choose
$\psi\equiv c$ and $p$ arbitrarily. However, if $c\in X$, then we could as well choose $p\equiv c$ and $\psi$ arbitrarily except $\psi(c)=c$.
The following result shows that we can always choose $\delta_f$ for $\psi$.

\begin{proposition}\label{prop:64585}
Let $f\colon X^n\to Y$ be a transformed polynomial function as described in (\ref{eq:Tr-pol7434}). Then, we have $f=\delta_{f}\circ p$. In
particular, if $p$ is a Sugeno integral, then we have $\psi=\delta_{f}$.
\end{proposition}

\begin{proof}
Since any polynomial function $p\colon X^n\to X$ is $\mathcal{R}_p$-idempotent, we have $\delta_f\circ p=\psi\circ\delta_p\circ p=\psi\circ
p=f$. The second part follows immediately since then $\delta_p=\mathrm{id}_X$ is the identity function on $X$.
\end{proof}

We also have the following result, which is the counterpart of Proposition~\ref{prop:Sugeno4565}.

\begin{proposition}\label{prop:Sug223}
A function $f\colon X^n\to Y$ is a transformed polynomial function if and only if it is a transformed Sugeno integral.
\end{proposition}

\begin{proof}
Clearly, any transformed Sugeno integral is a transformed polynomial function. Conversely, let $f\colon X^n\to Y$ be a transformed polynomial
function as described in (\ref{eq:Tr-pol7434}). Let $q\colon X^n\to X$ be a Sugeno integral such that $p=\langle q\rangle_p=\delta_p\circ q$
(cf.\ Fact~\ref{nondecreasing}). It follows that $f=\psi\circ p=\psi\circ\delta_p\circ q=\delta_f\circ q$ is a transformed Sugeno integral.
\end{proof}

Denote the domain of a function $f$ by $\mathcal{D}_f$. Recall that a function $h$ is a \emph{right-inverse} \cite[p.~25]{AlsFraSch06} of a
function $f$ if $\mathcal{D}_h=\mathcal{R}_f$, $\mathcal{R}_h\subseteq\mathcal{D}_f$, and $f\circ h=\mathrm{id}_{\mathcal{R}_f}$. It can be
shown that the statement ``every function has at least one right-inverse'' is equivalent to the axiom of choice. Even though, in general, we
have to appeal to the axiom of choice in order to ensure the existence of right-inverses, this requirement is not necessary in many concrete
situations, for instance, when dealing with monotone functions over the real numbers.

We say that $f$ is \emph{quasi-idempotent} if $\mathcal{R}_{\delta_f}=\mathcal{R}_f$. The terminology ``quasi-idempotent'' is justified by the
following result (see \cite{Mar} for the real case).

\begin{proposition}\label{prop:QI00}
Let $f\colon X^n\to Y$ be a function. Under the axiom of choice, $f$ is quasi-idempotent if and only if there is an idempotent function $g\colon
X^n\to X$ and a function $\psi\colon \mathcal{R}_g\to Y$ such that $f=\psi\circ g$. In this case, $\psi=\delta_f$.
\end{proposition}

\begin{proof}
Sufficiency is straightforward. We have $\mathcal{R}_{\delta_f}=\mathcal{R}_\psi=\mathcal{R}_{\psi\circ g}=\mathcal{R}_f$. For the necessity we
observe that, if $h\colon\mathcal{R}_{\delta_f}\to X$ is a right-inverse of $\delta_f$, then the function $g\colon X^n\to X$, defined by
$g(\vect{x})=x_1$ if $x_1=\cdots =x_n$ and $g(\vect{x})=(h\circ f)(\vect{x})$ otherwise, is idempotent and satisfies $f=\delta_f\circ g$.
\end{proof}

The following theorem yields a characterization of transformed polynomial functions whose diagonal sections are lattice homomorphisms. These
functions are clearly order-preserving by Fact~\ref{nondecreasing} and Proposition~\ref{prop:64585}.

\begin{theorem}\label{thm:QPfunctions}
Let $f\colon X^n\to Y$ be a function and assume that $\delta_f$ is a lattice homomorphism. Then, under the axiom of choice, $f$ is a transformed
polynomial function if and only if it is a quasi-idempotent quasi-polynomial function.
\end{theorem}

\begin{proof}
Suppose $f\colon X^n\to Y$ is a quasi-idempotent quasi-polynomial function. Let $h\colon\mathcal{R}_{\delta_f}\to X$ be a right-inverse of
$\delta_f$, i.e., $\delta_f\circ h=\mathrm{id}_{\mathcal{R}_{\delta_f}}$, and define $\psi=\delta_f$ and $p=p_{h\circ f}$, where $p_{h\circ f}$
is the unique polynomial function extending $h\circ f|_{\{0,1\}^n}$ (see Proposition~\ref{Goodstein}), that is, 
$$
p_{h\circ f}(\mathbf{x})=\bigvee_{I\subseteq [n]}\Big((h\circ f)(\mathbf{e}_I)\wedge\bigwedge_{i\in I}x_i\Big).
$$
Since $\delta_f$ preserves $\vee$ and $\wedge$, we have $\psi\circ p=\delta_f\circ p_{h\circ f}=p_f\circ\delta_f=f$, which shows that $f$ is a
transformed polynomial function.

Conversely, suppose $f\colon X^n\to Y$ is a transformed polynomial function. By Propositions~\ref{prop:Sug223} and \ref{prop:QI00}, $f$ is
quasi-idempotent. Moreover, by Proposition~\ref{prop:64585}, we have $f=\delta_f\circ p=p_{\delta_f\circ p}\circ\delta_f=p_f\circ\delta_f$,
which shows that $f$ is a quasi-polynomial function.
\end{proof}

\begin{remark}
\begin{enumerate}
\item[(i)] We have seen in the proof of Theorem~\ref{thm:QPfunctions} that a transformed polynomial function $f\colon X^n\to Y$ whose $\delta_f$
is a lattice homomorphism satisfies the equations $f=p_f\circ\delta_f=\delta_f\circ p_{h\circ f}$ for any right-inverse $h$ of $\delta_f$. In a
sense, the transformation and polynomial function commute since $p_{h\circ f}$ has just the same $\vee$-$\wedge$ form as $p_f$.

\item[(ii)] Quasi-idempotency is necessary in Theorem~\ref{thm:QPfunctions}. Indeed, the real quasi-polynomial function $f\colon [0,1]^2\to
[0,1]$, defined by $f=p\circ\varphi$, where $\varphi(x)=1$ if $x\in [\frac 12,1]$ and $0$ otherwise, and $p(x_1,x_2)=\mathrm{med}(x_1\wedge
x_2,\frac 12,x_1\vee x_2)$, is not quasi-idempotent since $f(1,0)=f(0,1)=\frac 12\notin \mathcal{R}_{\delta_f}$.

\item[(iii)] By combining Theorem~\ref{thm:QPfunctions} with existing characterizations of quasi-polyno\-mial functions (such as those presented
in Section 4), we immediately generate characterizations of transformed polynomial functions.
\end{enumerate}
\end{remark}

By combining Proposition~\ref{prop:5353452} and Theorem~\ref{thm:QPfunctions}, we obtain the following result which yields conditions under
which a transformed polynomial function is a polynomial function. Observe that, in Theorem~\ref{thm:QPfunctions}, the appeal to the axiom of
choice is used only to show that the conditions are sufficient.

\begin{proposition}\label{prop:535345452}
Let $f\colon X^n\to X$ be a transformed polynomial function. Then $f$ is a polynomial function if and only if it is $\mathcal{R}_f$-idempotent
and $\delta_f$ is a lattice homomorphism with a convex range.
\end{proposition}

\section{Concluding remarks and future work}

In this paper we considered quasi-polynomial functions as mappings defined and valued on, possibly different, bounded distributive lattices, and
not necessarily order-preserving. The relevance of this concept in fields such as decision making was made apparent by showing that monotone
quasi-polynomial functions coincide exactly with those overall preference functionals which can be factorized into Sugeno integrals applied to a
utility function. We provided several axiomatizations for this class of quasi-polynomial functions, which subsume those presented in
\cite{CouMar3}, and explicitly described all possible factorizations of a given quasi-polynomial function as a composition of a polynomial
function with a unary map. Moreover, we introduced the notion of transformed polynomial function as a natural counterpart of quasi-polynomial
function, and characterized the class of transformed polynomial functions accordingly. As it turned out, under the axiom of choice, those
transformed polynomial functions whose diagonal section is a lattice homomorphism constitute a proper subclass of quasi-polynomial functions.

Looking at natural extensions to this framework, we are inevitably drawn to consider the multi-sorted setting. More precisely, we are interested
in mappings $f\colon X_1\times \cdots \times X_n\to Y$ which can be factorized as a composition
$$
f(x_1,\ldots, x_n)=p(\varphi_1(x_1),\ldots, \varphi_n(x_n))
$$
where $p\colon Y^n\to Y$ is a polynomial function, and each $\varphi_i\colon X_i\to Y$ is a unary map defined and valued on, possibly different,
bounded distributive lattices $X_i$ and $Y$. These functions appear naturally within the scope of multicriteria decision making (see for
instance Bouyssou et al.~\cite{BouMarPir09}), and their axiomatization constitutes a topic of future research.


\begin{thebibliography}{10}

\bibitem{AlsFraSch06}
C.~Alsina, M.~J. Frank, and B.~Schweizer.
\newblock {\em Associative functions - Triangular norms and copulas}.
\newblock World Scientific Publishing Co. Pte. Ltd., Hackensack, NJ, 2006.

\bibitem{BelPraCal07}
G.~Beliakov, A.~Pradera, and T.~Calvo.
\newblock {\em Aggregation Functions: A Guide for Practitioners}.
\newblock Studies in Fuziness and Soft Computing. Springer, Berlin, 2007.

\bibitem{BouDubPraPir09}
D.~Bouyssou, D.~Dubois, H.~Prade, and M.~Pirlot, editors.
\newblock {\em Decision-Making Process - Concepts and Methods}.
\newblock ISTE/John Wiley, London, 2009.

\bibitem{BouMarPir09}
D.~Bouyssou, T.~Marchant, and M.~Pirlot.
\newblock A conjoint measurement approach to the discrete {S}ugeno integral.
\newblock In {\em The Mathematics of Preference, Choice and Order}, pages
  85--109. Springer-Verlag, Berlin, 2009.

\bibitem{BurSan81}
S.~Burris and H.~P. Sankappanavar.
\newblock {\em A course in universal algebra}, volume~78 of {\em Graduate Texts
  in Mathematics}.
\newblock Springer-Verlag, New York, 1981.

\bibitem{CouMar3}
M.~Couceiro and J.-L. Marichal.
\newblock Axiomatizations of quasi-polynomial functions on bounded chains.
\newblock {\em Aeq. Math.}, 78(1-2):195--213, 2009.

\bibitem{CouMar1}
M.~Couceiro and J.-L. Marichal.
\newblock Characterizations of discrete {S}ugeno integrals as polynomial
  functions over distributive lattices.
\newblock {\em Fuzzy Sets and Systems}, 161(5):694--707, 2010.

\bibitem{CouMar2}
M.~Couceiro and J.-L. Marichal.
\newblock Representations and characterizations of polynomial functions on
  chains.
\newblock {\em J. Mult.-Valued Logic Soft Comput.}, 16(1-2):65--86, 2010.

\bibitem{CouMar0}
M.~Couceiro and J.-L. Marichal.
\newblock Polynomial functions over bounded distributive lattices.
\newblock {\em J. Mult.-Valued Logic Soft Comput.}
\newblock In press.

\bibitem{DubFarPraSab09}
D.~Dubois, H.~Fargier, H.~Prade, and R.~Sabbadin.
\newblock A survey of qualitative decision rules under uncertainty.
\newblock In {\em Decision-Making Process - Concepts and Methods}, pages
  435--473. ISTE/John Wiley, London, 2009.

\bibitem{DubMarPraRouSab01}
D.~Dubois, J.-L. Marichal, H.~Prade, M.~Roubens, and R.~Sabbadin.
\newblock The use of the discrete {S}ugeno integral in decision-making: a
  survey.
\newblock {\em Internat. J. Uncertain. Fuzziness Knowledge-Based Systems},
  9(5):539--561, 2001.

\bibitem{Goo67}
R.~L. Goodstein.
\newblock The solution of equations in a lattice.
\newblock {\em Proc. Roy. Soc. Edinburgh Sect. A}, 67:231--242, 1965/1967.

\bibitem{GraMarMesPap09}
M.~Grabisch, J.-L. Marichal, R.~Mesiar, and E.~Pap.
\newblock {\em Aggregation functions}.
\newblock Encyclopedia of Mathematics and its Applications 127. Cambridge
  University Press, Cambridge, UK, 2009.

\bibitem{GraMurSug00}
M.~Grabisch, T.~Murofushi, and M.~Sugeno, editors.
\newblock {\em Fuzzy measures and integrals - Theory and applications},
  volume~40 of {\em Studies in Fuzziness and Soft Computing}.
\newblock Physica-Verlag, Heidelberg, 2000.

\bibitem{Grae03}
G.~Gr\"atzer.
\newblock {\em General lattice theory}.
\newblock Birkh\"auser Verlag, Berlin, 2003.
\newblock Second edition.

\bibitem{Mar01}
J.-L. Marichal.
\newblock An axiomatic approach of the discrete {S}ugeno integral as a tool to
  aggregate interacting criteria in a qualitative framework.
\newblock {\em IEEE Trans. Fuzzy Syst.}, 9(1):164--172, 2001.

\bibitem{Mar09}
J.-L. Marichal.
\newblock Weighted lattice polynomials.
\newblock {\em Discrete Mathematics}, 309(4):814--820, 2009.

\bibitem{Mar}
J.-L. Marichal.
\newblock Solving {C}hisini's functional equation.
\newblock {\em Aeq. Math.}, 79(3):237--260, 2010.

\bibitem{Rud01}
S.~Rudeanu.
\newblock {\em Lattice functions and equations}.
\newblock Springer Series in Discrete Mathematics and Theoretical Computer
  Science. Springer-Verlag London Ltd., London, 2001.

\bibitem{Sug74}
M.~Sugeno.
\newblock {\em Theory of fuzzy integrals and its applications}.
\newblock PhD thesis, Tokyo Institute of Technology, Tokyo, 1974.

\bibitem{Sug77}
M.~Sugeno.
\newblock Fuzzy measures and fuzzy integrals---a survey.
\newblock In M.~M. Gupta, G.~N. Saridis, and B.~R. Gaines, editors, {\em Fuzzy
  automata and decision processes}, pages 89--102. North-Holland, New York,
  1977.

\end{thebibliography}
\end{document}